\documentclass[a4paper,oneside]{amsart}

\usepackage{amsmath,amssymb,amsthm,amsfonts,amsrefs,enumerate,mathrsfs,mathtools,setspace,tikz,tikz-cd, fancyhdr}
\usepackage[colorlinks=true, linkcolor=blue, citecolor=blue]{hyper ref}
\usepackage[normalem]{ulem}
\usepackage[margin=1.5in]{geometry}

\relax

\theoremstyle{plain}
\newtheorem{theorem}[subsubsection]{Theorem}
\newtheorem{prop}[subsubsection]{Proposition}
\newtheorem{lemma}[subsubsection]{Lemma}

\newtheorem{corollary}[subsubsection]{Corollary}

\theoremstyle{definition}

\newtheorem{definition}[subsubsection]{Definition}
\newtheorem{remark}[subsubsection]{Remark}
\newtheorem*{unnumberedremark}{Remark}
\newtheorem*{lemma*}{Lemma}
\newtheorem*{prop*}{Proposition}
\newtheorem*{theorem*}{Theorem}
\newtheorem{definition*}{Definition}

\newcommand\coloneq{\mathrel{\mathop:}\mkern-1.2mu=}

\newcommand\N  { \ensuremath {\mathbb{N}}}
\newcommand\Z  { \ensuremath {\mathbb{Z}}}
\newcommand\A  { \ensuremath {\mathbb{A}}}
\newcommand\Zl  { \ensuremath {\mathbb{Z}_\ell}}

\newcommand\Q  { \ensuremath {\mathbb{Q}}}
\newcommand\Ql  { \ensuremath {\mathbb{Q}_\ell}}
\newcommand\R  {\ensuremath {\mathbb{R}}}
\newcommand\C  { \ensuremath {\mathbb{C}}}

\newcommand\EEightLattice  { \ensuremath {\mathbb{E}_8}}
\newcommand\HyperbolicPlaneLattice  { \ensuremath {\mathbb{H}_2}}

\newcommand\Spec { \ensuremath {\textup{Spec}}}
\newcommand\Pic {\ensuremath {\textup{Pic}}}
\newcommand\PicardLattice[2] {\ensuremath {{\textup{Pic}_{#1}}}}

\newcommand{\RelativeAbelianVariety}{A}
\newcommand{\UniversalRelativeAbelianVariety}[1]{\RelativeAbelianVariety_{#1}}
\newcommand{\AbelianVariety}{A}

\newcommand\End  { \ensuremath {\textup{End}}}

\newcommand\Tr  { \ensuremath {\textup{Tr}}}

\newcommand\et  {\textup{\'et}}

\newcommand\GaloisExtension[2]  { \ensuremath {\textup{Gal}(#1/#2)}}

\newcommand\AbsoluteGaloisGroup[1]  { \ensuremath {\textup{Gal}(\overline{#1}/#1)}}
\newcommand{\Places}{S}
\newcommand{\PlacesOfExtension}{\Extension{\Places}}
\newcommand\Extension[1]{\ensuremath{#1'}}

\newcommand\SOZ{\ensuremath{{\mathcal{G}_{\mathcal{L}}}}}
\newcommand\SOQ{{G_{\QLattice}}}

\newcommand{\wh} {\ensuremath{\widehat}}

\newcommand{\ra} {\ensuremath{\rightarrow}}
\newcommand{\ot} {\ensuremath{\otimes}}

\newcommand\btk  { \ensuremath {\begin{tikzcd}}}
\newcommand\etk  { \ensuremath {\end{tikzcd}}}
\newcommand{\ar} {\ensuremath{\arrow}}

\newcommand\Ring  { \ensuremath {R}}
\newcommand\Lattice  { \ensuremath {\Lambda}}
\newcommand\ZLattice  { \ensuremath {{\mathcal{L}}}}
\newcommand\QLattice  { \ensuremath {L}}
\newcommand\ZQuadraticForm  { \ensuremath {\mathcal{ Q}}}

\newcommand\NumberField  { \ensuremath {K}}
\newcommand\ExtensionOfNumberField  { \ensuremath {\Extension{\NumberField}}}
\newcommand\field{\ensuremath{k}}

\newcommand\GSpinZ{\ensuremath{\mathcal{GS}\textup{pin}}}
\newcommand\GSpinQ{\textup{GSpin}}
\newcommand\GSpZ{\ensuremath{\mathcal{GS}\textup{p}}}
\newcommand\GSpQ{\textup{GSp}} 
\newcommand\gencliff{\ensuremath{\mathcal{C}}}
\newcommand\gencliffmod{\ensuremath{\mathcal{W}}}

\newcommand{\Shim}[3] {\ensuremath {\textup{Sh}_{#2}}(#1, #3)}
\newcommand{\ShimZ}[3]{\ensuremath{\mathscr{S}_{#2}}(#1, #3)_{(p)}}
\newcommand{\Sh}[1]{\ensuremath {\textup{Sh}_{#1}}}

\newcommand{\KSQ}{\ensuremath{{i_\Q^{KS}}}}

\newcommand{\locsys}[1]{\ensuremath{{(#1)^{shf}}}}
\newcommand\kthreelat {\ensuremath{\mathcal{L}_{K3}}}
\newcommand\abspic {\ensuremath{N}}
\newcommand\co {\ensuremath{\mathbb{K}}}
\newcommand\cosp {\ensuremath{\mathbb{{K}'}}}
\newcommand\de  { \ensuremath {d}}
\newcommand\ktmod{\ensuremath{\mathcal{F}}}
\newcommand\ktmodde[1] {\ensuremath{\mathcal{F}_{#1}}}
\newcommand\ktmoddedettriv[1]{{\ensuremath{{\tilde{\mathcal{F}}_{#1}}}}}

\newcommand\ktmodderel[2] {\ensuremath{\mathcal{F}_{#1,#2}}}
\newcommand\ktmoddecorel[3] {\ensuremath{\mathcal{F}_{#1,#2,#3}}}

\newcommand\ktmoddecoQ[2]{\ensuremath{{F_{#1,#2}}}}

\newcommand\shafsetavpol[5] {\ensuremath{{\mathcal{A}_{#4,#5}(\mathscr{O}_{{#2},{#3}})}}}
\newcommand\shafsetav[4] {\ensuremath{{\textup{Shaf}_{#1}({#2},{#3},#4)}}}
\newcommand\shafsetkthreepol[4] {\ensuremath{{\ktmod_{#4}({#2},{#3})}}}
\newcommand\shafsetkthree[3] {\ensuremath{{\textup{Shaf}_{#1}({#2},{#3})}}}

\newcommand\ktsurf {\ensuremath{X}}
\newcommand\ktsurfext {\ensuremath{{X'}}}
\newcommand\ktsp {\ensuremath{\mathcal{X}}}
\newcommand\ktspext {\ensuremath{{\mathcal{X}'}}}
\newcommand\etaleisometries {\ensuremath{I}}
\newcommand\canbun {\ensuremath{\Omega^2}}
\newcommand\trivialbun {\ensuremath{\mathscr{O}}}

\newcommand\uktsurf {\ensuremath{\mathcal{X}}}
\newcommand\symsp {\ensuremath{\Omega}}
\newcommand\gspsymsp {\ensuremath{{\Omega'}}}

\newcommand\pol {\lambda}
\newcommand\polext {{\lambda'}}

\newcommand\geopt {{\overline{s}}}
\newcommand\lstr{\alpha}

\newcommand\GLZ{\mathcal{GL}}
\newcommand\GL{\textup{GL}}
\newcommand{\place}{\ensuremath{{\nu}}}

\newcommand\intmodel[2]{\mathscr{S}_{#1}(#2)_{(p)}}

\newcommand\isom{\ensuremath{\phi}}

\newcommand\discbound{\ensuremath{D}}

\newcommand\disc{\ensuremath{\textup{disc}}}

\newcommand\basesch{\mathcal{S}}

\newcommand\testsch{T}
\newcommand\ShafKthree[2]{\textup{Shaf}_{\textup{K3}}(#1,#2)}

\begin{document}

\title{The unpolarized Shafarevich Conjecture\\ for K3 Surfaces}
\author{Yiwei She}
\date{\today}
\maketitle

\tableofcontents
\subsection*{Acknowledgements} I thank Matthew Emerton for suggesting this problem and advising me in graduate school. I thank Keerthi Madapusi Pera for helpful conversations especially concerning the results of his papers which are used extensively.  I thank Brian Conrad for helpful and patient discussions.

\section{Introduction} 
We prove the unpolarized Shafarevich conjecture for K3 surfaces:  the set of isomorphism classes of K3 surfaces over a fixed number field $\NumberField$ with good reduction away from a fixed and finite set of places $\Places$ is finite. Our proof is based on the theorems of Faltings and Andr\'e, as well as the Kuga-Satake construction.

\subsection{Finiteness theorems}
Faltings ~\cite[Thm 6]{Fal84} proved the (polarized) Shafarevich conjecture for abelian varieties over number fields: 

\begin{theorem}[Faltings] \label{faltings} Let $\NumberField$ be a number field, $\Places$ be a  fixed finite set of places of $\NumberField$, and $n, \de \in \mathbb{N}$. Then the set
\[\shafsetavpol{AV}{\NumberField}{\Places}{n}{\de} \coloneq \left\{A/\NumberField \left| \begin{array}{l} (A,\pol)\textup{ is a polarized abelian variety which} \\ \textup{has dimension $n$,} \\ \textup{admits good reduction outside $\Places$,}\\ \textup{and $\deg \pol  = \de$.} \end{array} \right.\right\} /\sim_\NumberField\] is finite.
\end{theorem}

Polarized abelian varieties of any degree exist and are parametrized by a countably infinite set of moduli spaces. This leaves open the question of finiteness across  polarizations of all degrees $\de$.  More precisely, is the union 
\begin{equation}\label{unpolarized}  \shafsetav{AV}{\NumberField}{\Places}{n} \coloneq \bigcup_\de \shafsetavpol{AV}{\NumberField}{\Places}{n}{\de}\end{equation}
finite?    In ~\cite{Zar85}, Zarhin proved that the answer is yes: 

\begin{theorem}[Zarhin]\label{zar}
The set of isomorphism classes of abelian varieties over $\NumberField$, of dimension $n$, and having good reduction away from $\Places $ is finite; i.e. the set (\ref{unpolarized}) is finite.
\end{theorem}

Using the Kuga-Satake map (Section \ref{ks}) Andr\'e  ~\cite[Thm. 1.3.1]{And96} derived the analogue of Faltings' theorem for polarized K3 surfaces:

\begin{theorem}[Andre]\label{andre}
Let $\NumberField$ be a number field, $\Places$ a fixed finite set of its places, $\de\in \N$. The set  
\[\shafsetkthreepol{K3}{\NumberField}{\Places}{\de} \coloneq \left\{X/\NumberField \left| \begin{array}{l} (X,\pol)\textup{ a K3 surface which}  \\ \textup{ admits good reduction outside $\Places$}\\ \textup{ and $\deg \pol = \de$.} \end{array} \right.\right\}/\sim_\NumberField\]
is finite.
\end{theorem}
\begin{remark} By \emph{good reduction} of a K3 surface $\ktsurf$ (resp.  polarized K3 surface $(\ktsurf,\pol)$) at a place $\place \in \NumberField$ we mean that there exists a K3 space $\ktsp$ (resp. polarized K3 space $(\ktsp,\pol)$) over $\mathscr{O}_{\NumberField,\place}$ (see Defs. \ref{k3definition} and \ref{k3poldefinition}) with generic fiber $\ktsurf$ (resp. $(\ktsp,\pol)$).
\end{remark}
In this paper we prove the unpolarized Shafarevich conjecture.  This is the analogue of Zarhin's unpolarized finiteness theorem for K3 surfaces.

\begin{theorem}[Main theorem]\label{main} Let $\NumberField$ be a number field and  $\Places$ a fixed finite set of places of $\NumberField$.  The set 

\begin{equation}\label{unpolarizedk3}  \shafsetkthree{K3}{\NumberField}{\Places} \coloneq \bigcup_\de \shafsetkthreepol{K3}{\NumberField}{\Places}{\de}\end{equation}
of isomorphism classes of  K3  surfaces defined over $\NumberField$ with good reduction away from $\Places $ is finite.  
\end{theorem}

\section{K3 surfaces} 
\subsection{Definitions} We review the basic theory and definitions for K3  surfaces, closely following ~\cite{Riz06}.  

\begin{definition}\cite[1.1.1, 1.1.5]{Riz06} \label{k3definition} \begin{enumerate}
\item Let $\field$ be a field. A non-singular proper surface $\ktsurf/\field$ is a \emph{K3 surface} if $\canbun_{\ktsurf/\field} \simeq \trivialbun_{\ktsurf/\field}$ and $H^1(\ktsurf,\trivialbun_{\ktsurf/\field}) = 0$.

\item Let $\basesch$ be a scheme.  A \emph{K3 scheme} over $\basesch$ is a scheme $\ktsp $, with a proper smooth morphism $\pi \colon  \ktsp\to\basesch$ whose geometric fibers are K3 surfaces.  

\item Let $\basesch$ be a scheme.  A \emph{K3 space}  is an algebraic space $\ktsp $ with a proper smooth morphism  $\pi\colon  \ktsp \rightarrow \basesch$ and an etale cover (of schemes) 
$\basesch'\rightarrow \basesch$  such that $\pi'  \colon  \ktspext = \ktsp\times_\basesch \basesch' \rightarrow \basesch'$ is a K3 scheme.
\end{enumerate}
\end{definition}

To prove the main theorem, we will study the moduli spaces  $\ktmod_{\de,\Z}$ (see \ref{k3moduli}) of polarized K3 surfaces as well as some invariants of the second cohomology of the K3 surfaces.  First we review properties of $H^2$ and the Picard lattice. 
\begin{definition} An \emph{$\Ring$-lattice} is a free $\Ring$-module with a symmetric bilinear pairing. Maps of $\Ring$-lattices are the maps of $\Ring$-modules compatible with the pairing. We mean $\Z$-lattice when $R$ is omitted.
\end{definition} 


\subsection{Cohomology of K3 surfaces} Let $\ktsurf/\field$ be a K3 surface.  When $\field\subset \C$, the Betti numbers of $\ktsurf (\C)$ are  computed by Noether's formula:  $(b_0 , b_1,b_2,b_3,b_4) = (1,0,22,0,1)$. 
There is a perfect pairing given by the intersection product $\cup$ on $H^2(X,\Z) \simeq \Z^{22}$, which has been computed to be \[(H^2(\ktsurf,\Z),\cup) =-(\EEightLattice^2 \oplus\HyperbolicPlaneLattice^3)\] see ~\cite[12.3]{Sha96}.  By the $-$ sign, we mean negating the pairing, i.e. $-\HyperbolicPlaneLattice$ is the rank 2 hyperbolic lattice with pairing given by $\left(\begin{smallmatrix} 0&-1\\ -1 & 0 \end{smallmatrix}\right) $.
 We will refer to $\EEightLattice^2\oplus \HyperbolicPlaneLattice^3$ as the \emph{K3 lattice} and denote it by $\kthreelat $.


When  $\ktsurf/\field$ is a proper scheme over a field, the relative Picard functor $\Pic_{\ktsurf/\field}$ is representable by a  scheme, locally of finite type over $\field$ ~\cite[8.2, Thm 3]{BLR90}.   When $\ktsurf/\field$ is a K3 surface, this scheme is separated, smooth, 0-dimensional over $\field$, and the torsion subgroup $\Pic^{\tau}_{\ktsurf/\field} $ is trivial \cite[3.1.2]{Riz06}.

\begin{definition} Let  $\ktsurf/\field$ be a K3 surface.  We define the \emph{Picard lattice} of $\ktsurf$, denoted $\PicardLattice{\ktsurf}{}$, to be the abelian group $ \Pic_{\ktsurf/\field}(\field)$. The pairing is given by the intersection form.
\end{definition}
Note that $\PicardLattice{\ktsurf}{}$  contains the absolute Picard group $\Pic(X)$ as a finite index subgroup ~\cite[2.3.1]{And96}. 
For any abelian group $A$, we denote by $A_\ell$ the base change $A\otimes_\Z \Zl$, e.g. we have $(\PicardLattice{\ktsurf}{})_\ell =\PicardLattice{\ktsurf}{}\otimes \Zl$. 

For each $\ell$, the $\ell$-adic chern map defines an embedding of lattices
\[ 
c_\ell\colon   (  \PicardLattice{\ktsurf}{})_\ell \rightarrow H^2_\et(\ktsurf_{\overline \NumberField} , \Zl(1)) ^{\AbsoluteGaloisGroup{\NumberField}}
\]

If  $\field\subset \C$, then the analytic chern map defines an embedding of lattices
\[c\colon  \PicardLattice{\ktsurf}{}\rightarrow H^2(\ktsurf(\C) ,\Z(1)) \cap H^{1,1} (\ktsurf(\C),\C)\]

The Tate conjecture computes the image of the $\ell$-adic chern maps in the case when $k$ is a number field.  Although the first proof in the literature is due to Andr\'e (see \cite{And96}), we give a citation to a paper with a more precise statement.
\begin{theorem}[The Tate Conjecture for K3 surfaces {\cite[5.6]{Tat94}}]\label{tate} Let $\ktsurf/\NumberField$ be a  K3  surface over a number field. Then the map 
\[c_\ell \otimes \Ql\colon   (\PicardLattice{\ktsurf}{})_\ell \otimes_{\Zl} \Ql \rightarrow H^2_\et(\ktsurf_{\bar\NumberField} , \Ql(1)) ^{\AbsoluteGaloisGroup{\NumberField}}\]
is surjective (and is therefore an isomorphism).  
\end{theorem}

\begin{lemma} 
The images of the Picard lattice $\PicardLattice{\ktsurf}{}$  in  $H^2(\ktsurf(\C),\Z(1))$ and $H^2_\et(\ktsurf_{\overline{\NumberField}},\widehat{\Z}(1))$ under the respective chern maps are saturated. \end{lemma}
\begin{proof}The saturation of $\Pic_\ktsurf = \Pic_{\ktsurf/\NumberField}(\NumberField)$ in $H^2(\ktsurf(\C),\Z(1))$ follows from the fact that $\Pic_{\ktsurf_{\overline{\NumberField}}/\overline{\NumberField}} (\overline {\NumberField})$ is saturated in $H^2_{\et}(\ktsurf_{\overline{\NumberField}},\widehat{\Z} (1))$ and  the equalities \begin{align*} 
\Pic_{\ktsurf/\NumberField}(\NumberField) = \Pic_{\ktsurf/\NumberField} (\overline \NumberField)^{\AbsoluteGaloisGroup{\NumberField}} = \Pic(\ktsurf_{\overline{\NumberField}})^{\AbsoluteGaloisGroup{\NumberField} }  =  \Pic_{\ktsurf_{\overline{\NumberField}}/\overline\NumberField} (\overline \NumberField)^{\AbsoluteGaloisGroup{\NumberField} } .\end{align*} All the equalities in the above line result from the fact  $\Pic_{\ktsurf/\NumberField}$ and $\Pic_{\ktsurf_{\overline{\NumberField}}/\overline\NumberField}$ are representable functors.  The last term is clearly saturated in $ \Pic_{\ktsurf_{\overline{\NumberField}}/\overline{\NumberField}} (\overline {\NumberField})$.  \end{proof}

By the Artin comparison theorems, we have an isomorphism $H^2(\ktsurf(\C), \Z(1))\otimes\widehat{\Z}=  H^2_{\et} (\ktsurf, \widehat{\Z}(1))$ as $\widehat{\Z}$-lattices. This isomorphism is compatible with the chern maps and  intersection forms in that it restricts to isomorphisms of $\Zl$-lattices $(c(\PicardLattice{\ktsurf}{}))_{ \ell} \simeq  (c_\ell(\PicardLattice{\ktsurf}{}))_\ell$ for all $\ell$.

\subsection{Moduli of  K3  surfaces} \label{k3moduli}
We recall the construction and basic properties of moduli spaces of K3 surfaces following \cite{Riz06}.  

\begin{definition}\cite[3.2]{Riz06}
\begin{enumerate}
\item Let $\field$ be a field.   A \emph{polarization} on a K3 surface $\ktsurf/\field$  is a global section $\pol \in\Pic_{\ktsurf/\field}(\field)$ which over $\overline k$ is the class of an ample line bundle.  A polarization $\pol$ is \emph{primitive} if it is not a  non-trivial multiple of another  line bundle over $\overline{k}$.

\item Let $\basesch$ be a scheme and let  $\pi\colon \ktsp \rightarrow \basesch$ be a  K3 space.  A polarization on $\ktsp$ is  a global section $\lambda \in \Pic_{\ktsp/\basesch}(\basesch)$ such that for every geometric point $\geopt$ of $\basesch$, the section $\lambda_{\geopt} \in\Pic_{X_\geopt / \kappa(\geopt) }  (\kappa(\geopt))$ is a polarization of $\ktsurf_{\geopt}$.  A polarization $\pol$ on $\ktsp/\basesch$ is \emph{primitive} if $\pol_\geopt$ is primitive for every geometric point $\geopt$ on $\basesch$.
\end{enumerate}
\end{definition}

\begin{definition}\cite[4.3]{Riz06} \label{k3poldefinition} Fix $\de$ a natural number and let $R$ be a $\Z$-algebra. Let $\ktmodderel{\de}{R}$ be the \emph{moduli of degree $\de$-primitively polarized K3 surfaces}  defined by:  
\[\ktmodderel{\de}{R} (\basesch) = \left\{(\pi\colon \ktsp\rightarrow \basesch,\pol)\left| \begin{array}{l}\textup{$\basesch$ is an $R$-scheme, $\pi\colon \ktsp\rightarrow \basesch$ is a K3 space}\\ \textup{and $\lambda$ is a primitive degree $\de $ polarization on $\ktsp$}\end{array}\right.\right\}\]
\end{definition}

\begin{prop}[{\cite[4.3.3]{Riz06}}] The functor $\ktmod_{\de,\Z}$ 
is a separated Deligne-Mumford stack of finite type over $\Z$.
\end{prop}

To represent this moduli problem with schemes, we will need level structures.

\begin{definition}\label{discriminantkernel} Let $\ZLattice$ be a lattice over $\Z$.  The \emph{discriminant kernel} of $\ZLattice$ is the maximal open compact  subgroup of $\textup{SO}_\ZLattice(\widehat{\Z})$ which acts trivially on $(\ZLattice_{\widehat{\Z}})^\vee/\ZLattice_{\widehat{\Z}} = \ZLattice^\vee/\ZLattice$.
\end{definition}

For a polarized K3 surface $(\ktsurf,\pol)$ over $\NumberField$, there is an explicit description of the discriminant kernel of the primitive lattice.   Recall the K3 lattice \[\kthreelat = \EEightLattice^2\oplus\HyperbolicPlaneLattice^3.\] We fix a direct summand $\HyperbolicPlaneLattice\subset \kthreelat$ and a symplectic basis $e,f$ for $\HyperbolicPlaneLattice$ (i.e. $e^2=f^2 =0$ and  $ef = fe = 1$).  Let $v_\de = e - \de f$, a primitive vector of degree $2\de$.  The Betti realization  $(P^2(\ktsurf,\Z) , -\cup)$ of the primitive cohomology is isomorphic to $v_d^\perp \subset \kthreelat$.  Let $\ZLattice_\de$ be the $\Z$-lattice: 
\[\ZLattice_\de = \EEightLattice^2 \oplus \HyperbolicPlaneLattice^2\oplus \langle 2\de\rangle \simeq v_\de^\perp\]   
Let $ \textup{SO}_{\kthreelat}$ denote the  $\Z$-algebraic group of isometries of $\kthreelat$.  The discriminant kernel of $P^2(\ktsurf,\Z)$ can be identified with the subgroup of $g\in \textup{SO}_{\kthreelat}(\widehat{\Z})$ which fixes $v_\de$.   

\begin{definition}Let $R$ be a $\Z$-algebra. Let $\etaleisometries_\de$ be the sheaf on the \'etale site of a K3 space $\pi\colon \ktsp \rightarrow\basesch$ 
with sections given by
\[\etaleisometries(\basesch) = \left\{ \lstr\colon  \kthreelat\otimes \wh{\Z} \simeq R^2_{\et}\pi_* (\ktsp,\wh\Z(1)) \left| \begin{array}{l}\alpha \textup{ is an isometry such that }\\  \lstr(v_\de) = c_{\wh\Z}(\pol) \end{array}\right.\right\}.\]
 Here $c_{\widehat{\Z}}$ is the product of the $\ell$-adic chern maps $c_\ell$. Let $\co_\de $ be a compact open subgroup of the discriminant kernel of $\ZLattice_\de$.  Precomposition defines a $\co_\de$ action on  $\etaleisometries_\de(\basesch)$.   A \emph{$\co_\de$-level structure} on a  K3  space $(\pi\colon \ktsp\rightarrow \basesch,\pol)$ is a section $\lstr \in H^0_{\et}(\basesch,  \etaleisometries_\de/\co_\de)$.
\end{definition}

\begin{definition} Let $\de>0$ and $\co_\de $ be a compact open subset of the discriminant kernel of $\ZLattice_\de$.  The \emph{moduli functor  $\ktmod_{\de,\co_\de,R}$ of degree $\de$-primitively polarized K3 surfaces with level $\co_\de$ structure} is defined by
\[\ktmoddecorel{\de}{\co_\de}{R} (\basesch) = \left\{(\pi\colon \ktsp\rightarrow \basesch,\pol)\left| \begin{array}{l}\text{$\basesch$ is an $R$-scheme, $\pi\colon \ktsp\rightarrow \basesch$ is a K3 space,}\\ \text{$\lambda$ is a primitive degree $\de $ polarization on $\ktsp$,} \\ \text{and $\lstr$ is a $\co_\de$-level structure on $\ktsp$}\end{array}\right.\right\}\]
\end{definition}

\begin{prop}\cite[2.4.3]{Riz10}
\label{representable}
For a torsion free compact open $\co_\de$ of the discriminant kernel, the stack $\ktmoddecoQ{\de}{\co_\de}\coloneq \ktmoddecorel{\de}{\co_\de}{\Q} $ is representable by a quasi-projective scheme over $\Q$.
\end{prop}

\begin{prop}\cite[3.11]{MP15}
\label{integralrepresentable}
For a torsion free compact open $\co_\de$ of the discriminant kernel, the stack $\ktmoddecorel{\de}{\co_\de}{\Z_{(p)}} $ is representable by an algebraic space over $\Z_{(p)}$.  For $p\not | \,\, \de$, the space $\ktmoddecorel{\de}{\co_\de}{\Z_{(p)}}$ is smooth over $\Z_{(p)}$.
\end{prop}

\section{A uniform Kuga-Satake map}\label{ks}

We will use a Kuga-Satake construction to associate  abelian varieties to K3 surfaces in a uniform way.
The classical Kuga-Satake map ~\cite{KS67} associates a complex abelian variety $A_\ktsurf^{KS}$  to a complex polarized  K3  surface $\ktsurf$, such that there is an isomorphism of $\Z$-Hodge structures \cite[5.7]{Del72}
\[ C^+(P^2(\ktsurf,\Z ) (1)) \simeq \End_C( H^1(A^{KS}_\ktsurf,\Z ))  \hspace{1cm}\text{ \parbox{5cm}{(compatible with the Hodge filtration after tensoring with $\C$)}}\]   
where $C$ (resp. $C^+$) is the Clifford (resp. even Clifford) algebra associated to the lattice $P^2(\ktsurf,\Z)(1)$. 

This construction was shown to be compatible in families over $\Spec(\Z[\frac{1}{2}])$; the case of $\Z[\frac{1}{2\de}]$ is due to ~\cite{Riz05}, and that of $\Z[\frac{1}{2}]$ to ~\cite{MP15}. 
In this section we use this construction to define a variation of the classical Kuga-Satake map over $\Z[\frac{1}{2}]$ in order to prove our main theorem.

\begin{unnumberedremark} In ~\cite{And96} and ~\cite{Del72}, each $\ktmod_{\de}$ is embedded in a different moduli space of abelian varieties (the degree of the polarization of the abelian variety increases as a function of $\de$).   The goal of this construction is to embed $\ktmod_{\de}$ in the same moduli space (of abelian varieties) as $\de$ varies.
\end{unnumberedremark}   

\subsection{Orthogonal Shimura varieties} \label{orthosv} We recall the theory of Shimura varieties for certain orthogonal groups.  For more details and general theory of Shimura varieties, see ~\cite{Del79}, ~\cite{Mil05}, and ~\cite{Moo98}.  

Let $(\ZLattice,\ZQuadraticForm)$ be a $\Z$-lattice of rank $n+2$ and signature $(n+,2-)$.  Let $\SOZ  \coloneq  \mathcal{SO}( \ZLattice, \ZQuadraticForm)$ be the algebraic group of automorphisms over $\Z$, and let $\SOQ \coloneq  (\SOZ)_\Q$.  Let $\symsp$ be the space of negative definite planes in $\ZLattice \otimes  \R$.   Then $\symsp$ is identified with the space of weight zero Hodge structures on $\ZLattice \otimes \C$ as follows (for details see \cite[4]{MP15} or ~\cite[2]{Riz10}).  For $h\in \symsp$ with basis $(e_h, f_h)$, the corresponding decomposition is given by
\[(\ZLattice\otimes \C)_h^{p,q} = \left\{ \begin{array}{ll}\langle e_h+ i f_h\rangle & (p,q)=(-1,1);\\
																							h_\C^\perp = \langle e_h,f_h\rangle ^\perp   & (p,q) = (0,0) ;\\
																							\langle e_h-if_h\rangle & (p,q) = (1,-1)  \end{array}  \right.\]

  The pair $(\SOQ,\symsp)$ is a Shimura datum of weight 0. In our application  we have $n\geq 1$ so that the group $\SOQ $ is split over $\Q$ and the pair $(\SOQ, \symsp)$ has  reflex field $E(\SOQ, \symsp) = \Q$.  
   For a compact open $\co \subset \SOQ(\A_f)$, let $\Shim{\SOQ}{\co}{\symsp}$ be the canonical model associated to $(\SOQ,\symsp)$ and $\co$ over $ \Q$. We will assume $\co$ is torsion free so that the canonical model is a quasi-projective scheme.

  \begin{theorem}[{\cite[4.6]{MP15}}]  \label{integralmodel} Assume $(\ZLattice,\ZQuadraticForm)$  has cyclic discriminant of order $\de$ and let  $\co_\de$ be a subgroup of the discriminant kernel. For $p>2$, $\Shim{\SOQ}{\co_\de}{\symsp}$ admits an integral canonical model $\intmodel{\co_\de}{\QLattice} /  \Z_{(p)}$. 
\end{theorem}
  
    By the torsion freeness of $\co_\de$, the functor $\Shim{\SOQ}{\co_\de}{\symsp}$ (resp. $\intmodel{\co_\de}{\QLattice}$) is a quasi-projective scheme over $\Q$ (resp. quasi-projective algebraic space over $\Z_{(p)}$).  We refer to \cite[4.3]{MP16}  and \cite[3.3]{Moo98} for the definition of integral canonical models.

\subsection{The Torelli Theorem} We now apply the definitions above to construct period domains for K3 surfaces and to state the Torelli theorem.  

Let  $(\ktsurf,\lambda)$ be a degree $2\de$-primitively polarized  K3 surface over a field $\field \subset \mathbb{C}$.  Recall the notation of \ref{k3moduli}. We fix an isomorphism $(H^2(\ktsurf, \Z),\cup) \simeq\kthreelat =\EEightLattice^2 \oplus \HyperbolicPlaneLattice^3$ and identify the two lattices (here $\cup$ is the negative of the natural intersection product,  this negation does not affect any arguments).  Let $ch(\pol) \in \kthreelat$ be the chern class of the polarization under this identification.
We fix the summand $\HyperbolicPlaneLattice\subset \kthreelat$, a basis $ e,f\in \HyperbolicPlaneLattice$ and $v_\de$ as in \ref{k3moduli}. Then there is an isometry $\isom\in \mathcal{SO}(\kthreelat)(\Z)$ so that $\isom (\text{ch}(\pol)) = e - \de f $  ~\cite[Exp. IX, \S1, Prop. 1]{BBD85} and (automatically) restricting to an isometry  \[  \isom (\pol^\perp)  = (e-\de f)^\perp \subset \kthreelat  . \]   

Let $\ZLattice_\de$ be the $\Z$-lattice defined in \ref{k3moduli},
which has signature  $(19+,2-)$.  Let $\QLattice_\de = \ZLattice_\de \ot \Q$ and denote the associated $\Q$-algebraic group of automorphisms of $\QLattice_\de$ by $\SOQ_{\de}$.  Let $\symsp_{\de}$ be the space of negative definite planes in $\ZLattice_{\de}\otimes \R$, and let  $\co_\de\subset \SOQ_{\de} ( \A_f)$ be a compact open subgroup of the discriminant kernel (see \ref{discriminantkernel}). 

We denote the Shimura variety associated to  $(\SOQ_{\de}, \symsp_{\de}) $ over $\Q$ with $\co_\de$-level structure by $\Shim{\SOQ _{\de}}{\co_\de}{\symsp_{\de}}$.  Since $\ZLattice_\de$ has cyclic discriminant $\Z/2\de\Z$, we have an integral model $\intmodel{\co_\de}{\QLattice_\de}/\Z_{(p)}$ of $\Shim{\SOQ _{\de}}{\co_\de}{\symsp_{\de}}$ for each $p>2$ by \ref{integralmodel}.

\begin{prop}[The Torelli theorem] \label{torelli} Let $p>2$.  Let $\co_\de\subset \SOQ_\de(\widehat{\Z})$ be a compact open subgroup of the discriminant kernel.    Assume $\co_{\de,2}$ is sufficiently small so that $\co_\de$ is torsion free, and that $\co_\de^2$ coincides with the discriminant kernel away from 2.
 Then there is an open immersion \[\iota_{\co_\de}\colon \ktmod_{\de ,\co_\de, \Z_{(p)}} \rightarrow \intmodel{\co_\de}{\ZLattice_\de}\] 
\end{prop}
\begin{proof} This theorem is due to Piatetski Shapiro-Shafarevich over $\C$ in \cite[]{PSS71},  Rizov  over $\Q$ in \cite[3.9.1]{Riz10}, and Madapusi-Pera over $\Z_{(p)}, p>2$ in \cite[5.15]{MP15}. 
\end{proof}

\subsection{Embeddings of orthogonal Shimura varieties} We will use the following unimodular $\Z$-lattice to construct the uniform Kuga-Satake construction
\[\ZLattice = \EEightLattice^2\oplus \HyperbolicPlaneLattice^2\oplus \langle 1\rangle^5\]

 Let $\SOZ= \mathcal{SO}(\ZLattice)$ be the $\Z$-algebraic group of isometries of $\ZLattice$, and let $\QLattice=\ZLattice_\Q$,  $\SOQ \coloneq  (\SOZ)_\Q$.  Fix $\co \subset \SOQ(\A_f)$ a torsion free compact open contained in the discriminant kernel.  By \ref{integralmodel}, the Shimura variety $\Shim{\SOQ}{\co}{\symsp}$ admits an integral canonical model $\intmodel{\co}{L}$ over $\Z_{(p)}$, $p>2$.   We will construct embeddings $i_d\colon \ktmod_{\de, \co_\de,\Z_{(p)} }\rightarrow\intmodel{\co}{\QLattice}$.  
   
   \begin{lemma} \label{inclusions} For any $\de\in \N$ there exists a primitive embedding of lattices 
 \begin{equation*}  i _d\colon \ZLattice_d = \EEightLattice^2\oplus \HyperbolicPlaneLattice^2 \oplus \langle 2\de \rangle \hookrightarrow  \EEightLattice^2\oplus \HyperbolicPlaneLattice^2\oplus \langle 1\rangle^5 =\ZLattice.
 \end{equation*}
 \end{lemma}
\begin{proof} It suffices to construct a primitive embedding of  $\langle 2\de \rangle$ in $\langle 1\rangle ^5$.  We can express  $2\de = 1+z^2+w^2+v^2+u^2$ by Lagrange's theorem. Hence if $v$ is a generator of $\langle 2\de \rangle$, $v \mapsto (1, z, w, v,u)$ is a primitive metric embedding of lattices.
 \end{proof}
 
Let $\Lattice_\de = \ZLattice_\de^\perp \subset \ZLattice$.  We define the closed subgroup scheme $\SOZ_{\de}$  of $\SOZ$ by its points as follows.  For a $\Z$-algebra $R$, 
\[\SOZ_{\de} ( R) \coloneq  \{g \in \SOZ(R)  \,\,\colon  \,\,g|_{(\Lattice_{\de})_R}  = \text{Id}\}.\]  Note that $(\SOZ_{\de})_\Q$ is isomorphic to $\SOQ_{\de}$ from Section 2.  
The restriction map $r_\de\colon  \SOZ_\de \rightarrow \mathcal{SO}(\ZLattice_\de)$  is defined in the obvious way by their values on the $R$ points.   


\begin{prop}[{\cite[6.1]{MP16}}] \label{orthoembedding}The inclusions  in \ref{inclusions} define maps of Hodge structures $i_{\de}\colon  \symsp_\de \ra \symsp$, so that we have an embedding of Shimura data
$i_\de\colon  (\SOQ_{\de},\symsp_{\de}) \hookrightarrow (\SOQ,\symsp)$.
Let $\co\subset \SOQ(\A_f)$ be a fixed compact open.  Then for any $\co_\de$ with $i_\de (\co_\de) \subset \co$, we have a finite and unramified map $i_{\co_\de}\colon \Shim{\SOQ_{\de}}{\co_\de}{\symsp_\de} \rightarrow \Shim{\SOQ}{\co}{\symsp}$.
\end{prop}


The integral canonical model  $\intmodel{\co_\de}{\QLattice_\de}$ satisfies the extension property by definition.   See \cite[4.2, 4.3]{MP16} for the definition of  extension property.  By the extension property we have the next corollary.
\begin{corollary}
For each $p>2$, the map $i_{\de}$ in \ref{orthoembedding} extends to a map of integral models $ \intmodel{\co_\de}{\QLattice_\de} \rightarrow \intmodel{\co}{\QLattice} $ over $\Z_{(p)}$.
\end{corollary}

\subsection{A uniform Kuga-Satake morphism} 
We continue the notation from the previous section, in particular $\ZLattice = \EEightLattice^2\oplus \HyperbolicPlaneLattice^2\oplus \langle 1 \rangle^5$ and $\SOQ= \mathcal{SO}(\ZLattice)_\Q$.  
  The goal of this section is to construct a family of abelian varieties over $\Shim{\SOQ}{\co}{\symsp}$, i.e. a Kuga-Satake morphism for $\Shim{\SOQ}{\co}{\symsp}$.  The inclusions $i_d\colon  \SOQ_{\de}\rightarrow \SOQ$ will then  induce Kuga-Satake morphisms for $\Shim{\SOQ_{\de}}{\co_\de}{\symsp_\de}$  and $\ktmod_{\de, \co_\de , \Q}$.  


Let $\gencliff(\ZLattice)$ (resp. $\gencliff^+(\ZLattice)$) be the Clifford algebra (resp. even Clifford algebra) of $\ZLattice$. Let $\GSpinZ_\ZLattice$ be the spin group associated to $\ZLattice$.  This is the $\Z$-algebraic group which has $R$ points 
\[\GSpinZ_\ZLattice(\Ring)=\{g \in \gencliff^+_\Ring(\ZLattice_\Ring )^\times \,  |  \,g\ZLattice_\Ring  g^{-1}  = \ZLattice_\Ring \} .\] 
 where $\gencliff^+_\Ring(\ZLattice_\Ring)$ is the even Clifford algebra of $\ZLattice_\Ring$.  We refer to ~\cite[3]{Del72} for details  on Clifford algebras and Spin groups.


We denote $\gencliff(\ZLattice)$ by $\gencliffmod$ when we consider it as its own left module. The natural inclusion $i\colon  \GSpinZ \hookrightarrow \gencliff (\ZLattice)^\times$ defines a faithful representation $\GSpinZ_\ZLattice \ra \GLZ(\gencliffmod)$.  Let $\iota\colon  \gencliff(\ZLattice) \rightarrow \gencliff(\ZLattice)$ be the canonical anti-involution of $\gencliff(\ZLattice)$ (defined on a  basis of $\gencliff(\ZLattice)$ by  $e_1\cdot \ldots \cdot e_k \mapsto e_k\cdot \ldots \cdot e_1$ where $\{e_i\}_{n+2}$ is a basis of $\ZLattice$).  
Let $a \in \gencliff(\ZLattice)$ be a nonzero fixed point of $\iota$.  We have a bilinear pairing on $\gencliffmod$ defined by
\[\phi_a \colon  \gencliffmod \otimes \gencliffmod \ra \Z \hspace{1cm} \phi_a(x,y) \colon  = \Tr(\iota (x)ya).\]   The pairing $\phi_a$  is nondegenerate, alternating and invariant (up to the norm character) under the action of $\GSpinZ$ \cite[5.5]{Riz10}). 
Thus the map $\GSpinZ \ra \GLZ(\gencliffmod)$ factors through $ \GSpZ\coloneq \GSpZ(\gencliffmod,\phi_a)$.  We denote $\GSpZ_\Q$ by $\GSpQ$ and the space of Lagrangian subspaces of $(\gencliffmod_\R,\phi_a)$ by $\gspsymsp$. 

\begin{lemma}[{\cite[3.6]{MP16}}] \label{gsprep}  In the notation above: 
\begin{enumerate} 
\item We can choose $a\in \gencliff(\ZLattice)$ so that the embedding $\GSpinZ_\ZLattice\ra \GSpZ(\gencliffmod, \phi_a)$ induces an embedding of Shimura data $i^{ks} \colon (\GSpinQ_\QLattice,\symsp) \ra (\GSpQ, \gspsymsp)$. 
\item Let $\co$ be a compact open subset of the discriminant kernel, and let $\co'\subset \GSpQ(\mathbb{A}_f)$ be a compact open containing $\co$.  Then $i^{KS}$ induces a map of canonical models over $\Q$ 
\[ \KSQ \colon  \Shim{\GSpinQ_\QLattice}{\co}{ \symsp} \ra \Shim{\GSpQ}{\cosp}{\gspsymsp}\] 
\end{enumerate}
\end{lemma}



\noindent By explicit computation we have 
\begin{lemma}[{\cite[p.102]{Riz05}}] This map 
descends to a map (which we also denote by $\KSQ$)
\[\KSQ\colon  \Shim{\SOQ}{\co}{\symsp} \ra \Shim{\GSpQ}{\cosp}{\gspsymsp} \]
\end{lemma}

\begin{definition}
Over $\Shim{\GSpQ}{\cosp}{\gspsymsp}$, we have the universal abelian variety $(\RelativeAbelianVariety , \pol ,\lstr)$. We define the \emph{Kuga-Satake} abelian variety $\RelativeAbelianVariety ^{KS}_{\textup{Sh}_\co} \rightarrow \Shim{\SOQ}{\co}{\symsp}$ 
  to be the pullback of $\RelativeAbelianVariety $ under  $i^{KS}_\Q$.  
  \end{definition}
\begin{definition} \label{uniformkugasatake}Recall the embedding $i_\de \colon  \ZLattice_\de \ra \ZLattice$ (\ref{inclusions}) and its induced map $i_d\colon \Shim{\SOQ_\de}{\co_\de}{\symsp_\de} \rightarrow \Shim{\SOQ}{\co}{\symsp}$ (\ref{orthoembedding}).   Let $s_\de\colon  \testsch \rightarrow \Shim{\SOQ_\de}{\co_\de}{\symsp_\de}$ be a map corresponding to the polarized K3 surface with level structure $(\ktsurf_{s_\de},\pol_{s_\de},\lstr_{s_\de})$ over $\testsch$.  We define the uniform Kuga-Satake abelian variety (suppressing notation for polarization and level structure) $\RelativeAbelianVariety ^{uKS}_{\testsch}$  associated to $s_\de$ to be the pullback of $\RelativeAbelianVariety _{\Sh{\cosp}}$  under $s_\de\circ i_\de \circ i^{KS}$: 


\[\btk \RelativeAbelianVariety _{T}^{uKS} \ar[r] \ar[d] & \UniversalRelativeAbelianVariety {\Sh{\co_\de}}^{uKS} \ar[d] \ar[r, ] & \UniversalRelativeAbelianVariety {\Sh{\co}}^{KS} \ar[d, "f"]\ar[r] & \UniversalRelativeAbelianVariety {\Sh{\cosp}} \ar[d]
\\ \testsch \ar[r, "s_\de"] &\Shim{\SOQ_\de}{\co_\de}{\symsp_\de} \ar[r, "i_\de" ]& \Shim{\SOQ}{\co}{\symsp} \ar[r, "i^{KS}"] & \Shim{\GSpQ}{\cosp}{\gspsymsp} \etk\]

\end{definition}
%


\subsection{Cohomology}
The goal of this section is to establish Galois-equivariant $\ell$-adic realizations of the classical  cohomological relation $P^2(\ktsurf) \hookrightarrow \End(\UniversalRelativeAbelianVariety{\ktsurf}^{uKS})$.  We retain the notation of the previous section.

Let $f\colon  \UniversalRelativeAbelianVariety {\Sh{\co}}^{KS} \rightarrow \Shim{\SOQ}{\co}{\symsp}$ be the Kuga-Satake abelian variety constructed (in the third column of the diagram) above. The relative \'etale cohomology $R^1f_* (\Z_{\ell})$  defines a lisse $\Zl$-sheaf of rank $2^{\textup{rank}(\QLattice)} = 2^{25}$ on $\Shim{\SOQ}{\co}{\symsp}$. 


\begin{lemma}[{\cite[4.2]{MP15}}] Let $\Shim{\SOQ}{\co}{\symsp}$ and $\Shim{\SOQ_\de}{\co_\de}{\symsp_\de}$ be the Shimura varieties over $\Q$ constructed in the previous section. For all $\ell$, we have lisse $\Zl$-sheaves $\locsys{\ZLattice_{\Zl}}$ and $\locsys{\ZLattice_{\de,{\Zl}}}$  on $\Shim{\SOQ}{\co}{\symsp}$ and $\Shim{\SOQ_\de}{\co_\de}{\symsp_\de}$ respectively.  
\end{lemma}


\begin{definition} \label{ktmoddettriv}
We define $\ktmoddedettriv{\de}{}$ to be the degree two \'etale cover of $\ktmodde{\de}$ given by adding a trivialization $\det(L_\de)\otimes \Z_2 \simeq P^2_{\et}(\ktsp,{\Z}_2)$ to the data of $\ktmodde{\de}$.  We use the same notation for the degree two \'etale cover for all the other K3 moduli functors.  
\end{definition} 
\begin{remark} The cover $\ktmoddedettriv{\de}{} \to \ktmodde{\de}$ has a (non-canonical) section, and we may lift $\basesch$ points of the latter to $\basesch$ points of the former.  
\end{remark}

\begin{lemma}[{\cite[5.6.1]{MP15}}]\label{p2isototautologicalsheaf} 
 Let   $j_{\de,\co_\de, \Q} \colon \ktmoddecorel{\de}{\co_\de}{\Q} \ra \Shim{\SOQ_\de}{\co_\de}{\symsp_\de}$ be the Torelli map (see \ref{torelli}). For all $\ell$, we have an isomorphism of lisse $\Zl$-sheaves \[j_{\de, \co_\de,\Q}^*(\locsys{\ZLattice_{\de,\Zl}} ) \simeq  P^2_{\et}(\uktsurf,\Zl(1))\]
\end{lemma}

\begin{lemma} \label{kugasatakelattice}
If $\ZLattice$ is unimodular, then $\locsys{\ZLattice_{\Zl}}$ is a sub-$\Zl$-sheaf of $\End_{C(\ZLattice)}(R^1f_* (\Z_{\ell}))$.
\end{lemma}
\begin{proof} The statement needed here is found in \cite[4.2]{MP15}, which refers to \cite[3.3-3.12]{MP16} for a more detailed proof.  The proof in loc. cit. outlines the construction for $\locsys{\ZLattice_{\Ql}} $ associated to $\ZLattice_{\Ql}$. However all but one step go through for a $\Z$-lattice $\ZLattice$. The only place where $\Q$-coefficients are needed is the use of Lemma 1.4 in loc. cit. 3.12.  This lemma uses division by $\disc(\ZLattice)$  to construct a projector $\mathbf{\pi} \colon  \End(C(\ZLattice))_\Q \rightarrow \ZLattice_\Q$, however under the additional assumption that $\ZLattice$ is unimodular, it holds without tensoring with $\Q$.
\end{proof}

\begin{lemma}[{\cite[4.9]{MP15}}, {\cite[7.15]{MP16}}] \label{projectionlemma} Recall the embedding $i_\de \colon  \ZLattice_\de \rightarrow \ZLattice$, and the induced map of Shimura varieties $i_\de\colon \Shim{\SOQ_\de}{\co_\de}{\symsp_\de}\rightarrow \Shim{\SOQ}{\co}{\symsp}$.   

\begin{enumerate}
\item We have a metric embedding  $\locsys{\ZLattice_\de}\ra i_\de^* (\locsys{\ZLattice})$.
\item The  sublocal system $(\locsys{\ZLattice_\de})^\perp \subset  i_\de^* (\locsys{\ZLattice})$  is trivial.
\end{enumerate}
\end{lemma}


%


\begin{lemma}[\cite{MP15}, 5.6.(3)] \label{4.14} Let $s\colon \NumberField  \rightarrow \Shim{\SOQ_\de}{\co_\de}{\symsp_\de}$ be a $\NumberField$-point and let $(\ktsurf_{s}, \pol_{s}, \lstr_{s})$ be the corresponding tuple of  K3  surface, polarization, and level structure.  We have a Galois equivariant isomorphism $(\locsys{\ZLattice_\de} )_s \simeq P^2_{\et}(\ktsurf, \Zl(1))$.
\end{lemma}


\noindent Combining the above lemmas, we have: 
\begin{prop}\label{latticeembedding}  Let $s\colon \NumberField \ra \Shim{\SOQ_\de}{\co_\de}{\symsp_\de}  $ be a $\NumberField$-point. Let $(\ktsurf_{s}, \pol_{s}, \lstr_{s})$ be the corresponding tuple of  K3  surface, polarization, and level structure.  Then $P^2_{\et}(\ktsurf,\widehat{\mathbb{Z}}(1)) \simeq s^*(\locsys{\ZLattice_\de})$. 
The $\widehat{\mathbb{Z}}$-lattice $P^2_{\et}(\ktsurf, \widehat{\mathbb{Z}}(1))$ is  isometrically and $\AbsoluteGaloisGroup{\NumberField}$-equivariantly  embedded in $i_\de^*(\locsys{\ZLattice})_s$.  Its orthogonal complement $P^2_{\et}(\ktsurf, \widehat{\mathbb{Z}}(1))^\perp \subset (\locsys{\ZLattice})_s$ is trivial as a $\AbsoluteGaloisGroup{\NumberField}$ module.
\end{prop}


\section{Proof of theorem}
\setcounter{subsection}{1}
\setcounter{subsubsection}{0}

In this section we prove the main theorem: 

\begin{theorem}\label{mainthmend}
Let $\NumberField$ be a fixed number field and  $\Places$ a fixed set of places of $\NumberField$.  The set of isomorphism classes of  K3  surfaces defined over $\NumberField$ with good reduction away from $\Places $ is finite.  We denote this set by $\ShafKthree{\NumberField}{\Places}$.
\end{theorem}

First we prove that the minimal degree of a polarization on a K3 surface $\ktsurf$ is a function of its Picard lattice $\Pic_\ktsurf$ as a lattice.

\begin{prop}\label{poldependsonlattice} Fix a number field $\NumberField$ and let  $\abspic \subset \kthreelat$ be a primitive sublattice of the K3 lattice with signature $((\textup{rank} (\abspic)-1)+,1-)$.   Let  
\begin{equation}C_\abspic\coloneq  \left\{ v\in \abspic  \left|  \begin{array}{c}v^2 > 0 \textup{ and}  \\ \textup{ $\forall w \textup{ s.t. } w^2 = -2\textup{, } w*v \neq 0 $}\end{array}\right.\right\} 
\end{equation} 
Let $D_\abspic\coloneq  \inf_{v\in C_\abspic} v^2$. 
Then $C_\abspic$ is nonempty, so that $D_\abspic<\infty$.  Any K3 surface $\ktsurf/\NumberField$ with Picard lattice $ \PicardLattice{\ktsurf}{}  \simeq\abspic $ admits a polarization of degree $D_\abspic$.  
\end{prop}

\begin{proof} The proof is the same as \cite[2.3.2]{LMS14} and mostly due to Ogus, the only difference is that we are working with $\PicardLattice{\ktsurf}{} =\Pic_{\ktsurf/\NumberField}(\NumberField)$ instead of the Picard group $\Pic(X)$.  A polarization $\lambda\in \Pic_\ktsurf$ is ample iff a multiple $\lambda^n\in \Pic(X)$ is ample.  Ampleness in $\Pic(X)$ is checked by the Nakai-Moishezon-Kleiman (NMK) criterion  \cite[Sec 8, Thm 1.2]{Huy14}.  By adjunction, effective $\NumberField$-rational integral curves $C \subset \ktsurf$ have $[C]^2 \geq -2$.  Since $\ktsurf$ is minimal it suffices to check the NMK criterion with respect to $-2$ curves. Let $w\in \abspic$ be any $-2$-class (not necessarily effective).  

Reflection across $w$ 
is an isometry. We denote the group generated by these reflections and $-id$ by $R_N$.  Let \[V_\abspic  \coloneq  \abspic\otimes \R -\bigcup_{w^2=-2} w^\perp\]  be the open subset (see ~\cite[1.10]{Ogu83}) of $\abspic\otimes \R$ where intersection with any integral $-2$-class is non-zero.  By \cite[1.10]{Ogu83}, $V_N$ is a union of cones.  The group $R_N$ acts transitively on the set of topologically connected components of $V_N$, and each component has non-empty intersection with the lattice $N$. Since one of the components of $V_N$ is the ample cone, the transitivity of the isometry group $R_N$ implies the result.
\end{proof}

This implies the more or less straightforward corollary: 

\begin{corollary} \label{mainprop}Fix $\NumberField,\Places$ as in \ref{mainthmend}, and additionally fix a finite set of lattices $\{N_i\}_{i\leq n}$.  The subset of K3 surfaces $\ktsurf\in \ShafKthree{\NumberField}{\Places}$ which have Picard lattice $\Pic_X \simeq N_i$ for some $i$ is finite.

\end{corollary}

\begin{proof} The proposition \ref{poldependsonlattice} shows that a K3 surface $\ktsurf$ with Picard lattice $\Pic_X \simeq N_i$ admits a polarization whose degree is a function of the lattice $N_i$.  Hence if $\Pic_\ktsurf\simeq N_i$, we have that $\ktsurf$ admits some polarization of bounded degree.  Andr\'e's theorem shows that the set of K3's with good reduction outside $\Places$ and having a polarization of bounded degree is finite.
\end{proof}

We return to the proof of \ref{mainthmend}. We will need to make a base change to get level structures to use the Kuga-Satake morphism. We will prove the finiteness of the set 
\begin{equation} \label{finitenessafterbasechange} \left\{\ktsurfext\in \ShafKthree{\ExtensionOfNumberField}{\PlacesOfExtension} \left| \ktsurfext = \ktsurf \otimes \ExtensionOfNumberField\textup{ for some }\ktsurf \in \ShafKthree{\NumberField}{\Places} \right\}\right.\end{equation} 
where $\ExtensionOfNumberField/\NumberField$ is any Galois extension and where $\PlacesOfExtension$ is the set of all places of $\ExtensionOfNumberField$ over $\Places$.

The next lemma shows that it is enough to prove the finiteness of (\ref{finitenessafterbasechange}). 

\begin{lemma}  \label{finite extension}  Fix $\NumberField,\Places$ and let $\ExtensionOfNumberField/\NumberField$ be any finite Galois extension.  The finiteness of (\ref{finitenessafterbasechange}) implies Theorem \ref{mainthmend}.  


\end{lemma}

\begin{proof}Let $\Gamma= \GaloisExtension{\ExtensionOfNumberField}{\NumberField}$, $\ktsurf/\NumberField$ be an element of $\ShafKthree{\NumberField}{\Places}$, and set $\ktsurfext \coloneq  \ktsurf_{\ExtensionOfNumberField}$.  
We have that $\Gamma$ is a finite group acting on $\Pic_{\ktsurfext} = \Pic_{\ktsurfext/\ExtensionOfNumberField}(\ExtensionOfNumberField)$ by isometries of the lattice. The fixed sublattice of this action is  $\Pic_\ktsurf = \Pic_{\ktsurf/\NumberField}(\NumberField)$.  The set of conjugacy classes of finite subgroups of the $\Z$-points of any linear algebraic group, e.g. $\text{SO}(\Pic_{\ktsurfext})$, is finite ~\cite[5a]{Bor63}. Thus the set of conjugacy classes of images of $\GaloisExtension{\ExtensionOfNumberField}{\NumberField}$ in $\text{SO}(\Pic_{\ktsurfext})$ is finite, hence the set of possible Picard lattices $\{\Pic_{\ktsurf}| \ktsurf \in \ShafKthree{\NumberField}{\Places}\}$ is finite. Now apply \ref{mainprop}.
\end{proof}

Now we specify the extension $\ExtensionOfNumberField/\NumberField$.  Let $\co_\de$ be the discriminant kernel of $\QLattice_\de$. The subgroup $\co_\de(4) \subset \co_\de $ defines an \'etale cover $\Shim{\SOQ_{\de}}{\co_\de(4)}{\symsp} \rightarrow  \Shim{\SOQ_{\de}}{\co_\de}{\symsp}$, which extends to an \'etale cover of the integral models away from 2.  Since $\co_\de(4)$ is torsion free, $\Shim{\SOQ_{\de}}{\co_\de(4)}{\symsp} $ is a quasi-projective scheme over $\Q$.  Let $K'/\Q$ be the maximal extension of degree at most $ \#|\textup{SO}_{\kthreelat}(\Z/4\Z)| $, and unramified away from $2$.  This is a finite extension by Hermite-Minkowski. Moreover the $\ExtensionOfNumberField$ base extension $(\ktsurfext,\polext)/\ExtensionOfNumberField$ of a polarized K3 surface $(\ktsurf, \pol)/\NumberField$ has level $\co_\de(4)$ structure and is represented by a map of schemes $s\colon  \Spec(\ExtensionOfNumberField) \rightarrow \Shim{\SOQ_{\de}}{\co_\de(4)}{\symsp}$. 
We will prove the finiteness of (\ref{finitenessafterbasechange}). Using Lemma \ref{finite extension}, we assume from now on that $\NumberField = \ExtensionOfNumberField$. 



\begin{lemma}[{\cite[Ch. 9, Thm. 1.1]{Cas78}}] The set of isomorphism classes of lattices of bounded rank and discriminant is finite. 
\end{lemma}

The rest of the proof consists of bounding this discriminant using the given hypotheses of \ref{mainthmend}, that $\NumberField$ and $\Places$ are fixed.  We will do this by applying a Kuga-Satake construction to Faltings' finiteness theorem for abelian varieties.


\begin{prop}[Faltings]\label{restatedfaltings} The Shimura variety $\Shim{\GSpQ}{\cosp}{\gspsymsp}$ has finitely many $\NumberField$ points with good reduction outside of $S$.   
\end{prop}

\begin{proof} The variety $\Shim{\GSpQ}{\cosp}{\gspsymsp} $ parametrizes tuples $(A,\pol,\lstr)$
. By the (polarized) Shafarevich conjecture (cf. \cite{Fal84}, \cite[18.1]{Mil86}), we have that this set is finite.  For any fixed $\co'$, the set of $\co'$-level structures over $\NumberField$ is finite.
\end{proof}

\noindent We denote the finite set of abelian varieties from the above proposition by $(\AbelianVariety_i , \pol_i,\lstr_i)$, $i\leq n$. \\

 Let $(\ktsurf, \pol,\co_\de)$ be a $\NumberField$-point $s \rightarrow \ktmod_{\de, \co_\de}$ for any $\de$.  The Torelli map \ref{torelli} and the embeddings \ref{orthoembedding} and \ref{gsprep} define a point $s\rightarrow \Shim{\GSpQ}{\cosp}{\gspsymsp}$ corresponding to the uniform Kuga-Satake abelian variety $(\AbelianVariety^{uKS}_s, \pol^{uKS}_s, \lstr^{uKS}_s )$.  

\begin{prop} \label{goodreduction}  Let $s\colon \Spec(\NumberField) \rightarrow \ktmoddecoQ{\de}{\co_\de} (\NumberField)$ be the $\NumberField$-point corresponding to a K3 surface $(\ktsurf_s,\pol_s, \lstr_s)$ over $\NumberField$ such that $\ktsurf_s$ has good reduction away from $\Places$.  Then its associated uniform Kuga-Satake abelian variety $A^{uKS}_s$ has good reduction away from $\Places$.  

\end{prop}
\begin{proof} 

Let $\nu\not\in \Places$ be a place of $\NumberField$ over $p$ and where $(\ktsurf,\lambda,\lstr)$ has  good reduction, i.e. there  is a local model $(\ktsp , \pol)  / \mathscr{O}_{\NumberField,\nu}$.  Let $s \rightarrow \ktmod_{2d, \co_\de,\mathscr{O}_{\NumberField,\nu}}$ be the $\mathscr{O}_{\NumberField,\nu}$-point representing this model.  Composition with 
\[  \ktmod_{2d, \co_\de,\Z_{(p)}} \rightarrow \intmodel{\co_\de}{\QLattice_\de} \rightarrow \intmodel{\co}{\QLattice} \rightarrow  \ShimZ{\GSpQ}{\cosp}{\gspsymsp} 
\] defines a polarized abelian variety $(\AbelianVariety_s, \pol_s,\lstr_s)$ which is a model over $\mathscr{O}_{\NumberField,\nu}$.  Here $\ShimZ{\GSpQ}{\cosp}{\gspsymsp}$ is the local model of $\Shim{\GSpQ}{\cosp}{\gspsymsp}$.
\end{proof}

\begin{corollary}\label{finiteabelianvarieties} For any $\de$, let $s \in \Shim{\SOQ_\de}{\co }{\symsp_\de}$ be the point corresponding to a polarized K3 surface with good reduction away from $\Places$. Then $(\AbelianVariety^{KS}_s, \pol_s^{KS}, \lstr_s^{KS})$ is isomorphic to $(\AbelianVariety_i, \pol_i,\lstr_i)$ for some $i$.
\end{corollary}

\begin{remark}  Since $|\disc(\PicardLattice{\ktsurf}{})|_\ell = |\disc(c_\ell(\PicardLattice{\ktsurf}{}))|_\ell$ we can compute the $\ell$-adic valuation of $\disc(\PicardLattice{\ktsurf}{})$ from its image in $\ell$-adic cohomology.  Doing so for all $\ell$ will allow us to bound $\disc(\PicardLattice{\ktsurf}{})$. \end{remark}

\begin{definition} Let $\ktsurf /\NumberField$ be a surface over a number field $\NumberField$. The lattice of \emph{transcendental cycles} $T(\ktsurf)$ is the orthogonal complement of the image of $c\colon \PicardLattice{\ktsurf}{} \rightarrow H^2(\ktsurf(\C),\Z(1))$
\[T(\ktsurf) \coloneq  \PicardLattice{\ktsurf}{}^\perp \subset H^2(\ktsurf(\C),\Z(1)).\]
\end{definition}

\begin{prop} Let $\NumberField,\Places$ be as above.  There is a finite set \[\rho_i \colon  \AbsoluteGaloisGroup{\NumberField} \to \GL (M_i)\]  of $\AbsoluteGaloisGroup{\NumberField}$  representations with coefficients in $\wh{\Z}$  such that for any $\ktsurf/\NumberField$ satisfying the hypotheses of \ref{mainthmend}, we have an isometry \[T ( \ktsurf)\otimes\widehat{\Z} \simeq M_i\] of $\AbsoluteGaloisGroup{\NumberField}$ representations for some $i$. 
\end{prop}

\begin{proof}
Let $(\AbelianVariety_i,\pol_i,\lstr_i)$ be the finite set of polarized abelian varieties from \ref{restatedfaltings}.
For a fixed $\AbelianVariety_i$, the set of $\phi\colon C(\ZLattice)\to \End(H^1(\AbelianVariety_i,\Z)) $ is finite by \cite[1.3.3.7]{Lan13}.  Since $(\ZLattice_{i})_{\widehat{\Z}} \subset \End_{C(\ZLattice)}(H^1_{\et}(\AbelianVariety_i,\widehat{\Z}))$ is determined by $\phi$, it suffices to show that each $(\ZLattice_{i})_{\widehat{\Z}}$ determines finitely many Galois representations $M_i$ in the list above. In fact, each $(\ZLattice_{i})_{\widehat{\Z}}$ determines one Galois representation \[M_i = \left(((\ZLattice_{i})_{\widehat{\Z}})^{\AbsoluteGaloisGroup{\NumberField}}\right)^\perp\] where the orthogonal complement is taken in the lattice $(\ZLattice_{i})_{\widehat{\Z}}$.

By Lemmas \ref{finiteabelianvarieties} and \ref{latticeembedding}, for any point $(\ktsurf,\pol,\lstr )$ of $\Shim{\SOQ_\de}{\co_\de}{\symsp_\de}$, we have Galois equivariant, metric embeddings
\[i_\de\colon  P^2_\et(\ktsurf , \widehat{\Z}(1)) \rightarrow   (\ZLattice_{i})_{\widehat{\Z}}\rightarrow \End_{C(\ZLattice)} (H_{\et}^1 (\AbelianVariety_i ,\widehat{\Z}))\] 
where $(\ZLattice_{i})_{\widehat{\Z}}$ is the $\AbsoluteGaloisGroup{\NumberField}$-submodule of $ \End_{C(\ZLattice)}(H^1_{\et} (\AbelianVariety_i,\widehat{\Z}))$ defined by \ref{kugasatakelattice}.  The Tate conjecture for K3  surfaces over number fields shows that the transcendental lattice $T(\ktsurf)_{\widehat{\Z}}$ is isomorphic to $(P^2_{\et}(\ktsurf,\widehat{\Z}(1))^{\AbsoluteGaloisGroup{\NumberField}})^\perp $ (the orthogonal complement taken in $P^2_{\et}(\ktsurf,\widehat{\Z}(1))$).  

We have that 
$P^2_\et(\ktsurf , {\widehat{\Z}}(1)) ^\perp \subset (\ZLattice_{i})_{\widehat{\Z}}$ is  trivial as a $\AbsoluteGaloisGroup{\NumberField}$ representation by \ref{latticeembedding}. Hence  $T(X)\otimes \widehat{\Z} \simeq\left( ((\ZLattice_{i})_{\widehat{\Z}})^{\AbsoluteGaloisGroup{\NumberField}}\right)^\perp$.
\end{proof}

\begin{corollary}\label{transcendentaldiscriminant}For $\NumberField, \Places$ as above, there is an upper bound $\discbound$ such that for any $\ktsurf$, we have $\disc (T(\ktsurf)) \leq \discbound$.  This bound is determined by the set $(A_i,\pol_i,\lstr_i)$, $i\leq n$ so that $\discbound$ depends only on $\NumberField$ and $\Places$.
\end{corollary}

\begin{proof} By the  proposition above we have $T(X) \otimes \widehat{\Z} \simeq M_i$ as a $\widehat{\Z}$-lattice, in particular $\disc(T(X)) = \disc( M_i)$. Since the set of $M_i$ is finite, the conclusion follows.
\end{proof}


\begin{corollary} For $\ktsurf$ satisfying the hypotheses of Theorem \ref{mainthmend}, the discriminant of $\PicardLattice{\ktsurf}{}$ is bounded as a function depending only on $\NumberField$ and $S$.  
 \end{corollary}
 
\begin{proof} For two mutually orthogonal and saturated sublattices $N$ and $N'$, whose sum is finite index in a unimodular lattice $M$, we have
\[\disc(N) = \left| \frac{M}{N+N'}\right| =\disc (N').\]
This applies to $\Pic_\ktsurf$ and $T(X)$ as sublattices of $H^2(\ktsurf(\C), \Z(1))$.  Hence the discriminant of $\Pic_X$ is equal to the discriminant of $T(X)$, which is constrained to a finite set by \ref{transcendentaldiscriminant}.
 \end{proof}

\begin{corollary}
The set of isomorphism classes of lattices $\Pic_\ktsurf$ where $\ktsurf $ satisfies the hypotheses of \ref{mainthmend} is finite.  The main theorem \ref{mainthmend} now follows from \ref{mainprop}.
\end{corollary}

 

 
 


\section{References}
\begin{biblist}*{labels={alphabetic}}
\bib{And96}{article}{
   author={Andr{\'e}, Yves},
   title={On the Shafarevich and Tate conjectures for hyper-K\"ahler
   varieties},
   journal={Math. Ann.},
   volume={305},
   date={1996},
   number={2},
   pages={205--248},
   issn={0025-5831},
   review={\MR{1391213 (97a: 14010)}},
   doi={10.1007/BF01444219},
}

\bib{BBD85}{book}{
  author={Beauville, A.}
  author={Bourguignon, J.-P. }
  author={Demazure, M.},
  title={G\'eom\'etrie des {S}urfaces {K}3: {M}odules et {P}\'eriodes},
  series={Ast\'erisque},
  volume={126},
  publisher={Soc. Math. de France},
  year={1985},
  }
 
\bib{BLR90}{book}{
   author={Bosch, Siegfried},
   author={L{\"u}tkebohmert, Werner},
   author={Raynaud, Michel},
   title={N\'eron models},
   series={Ergebnisse der Mathematik und ihrer Grenzgebiete (3) [Results in
   Mathematics and Related Areas (3)]},
   volume={21},
   publisher={Springer-Verlag, Berlin},
   date={1990},
   pages={x+325},
   isbn={3-540-50587-3}
}

\bib{Bor63}{article}{
   author={Borel, Armand},
   title={Arithmetic properties of linear algebraic groups},
   conference={
      title={Proc. Internat. Congr. Mathematicians},
      address={Stockholm},
      date={1962},
   },
   book={
      publisher={Inst. Mittag-Leffler, Djursholm},
   },
   date={1963},
   pages={10--22},
   review={\MR{0175901}},
}
   
   
\bib{Cas78}{book}{
   author={Cassels, J. W. S.},
   title={Rational quadratic forms},
   series={London Mathematical Society Monographs},
   volume={13},
   publisher={Academic Press, Inc. [Harcourt Brace Jovanovich, Publishers],
   London-New York},
   date={1978},
   pages={xvi+413},
   isbn={0-12-163260-1},
   review={\MR{522835}},
}
     
\bib{Del72}{article}{
   author={Deligne, Pierre},
   title={La conjecture de Weil pour les surfaces  K3 },
   language={French},
   journal={Invent. Math.},
   volume={15},
   date={1972},
   pages={206--226},
   issn={0020-9910},
   review={\MR{0296076 (45 \#5137)}},
}



\bib{Del79}{article}{
   author={Deligne, Pierre},
   title={Vari\'et\'es de Shimura: interpr\'etation modulaire, et techniques
   de construction de mod\`eles canoniques},
   language={French},
   conference={
      title={Automorphic forms, representations and $L$-functions (Proc.
      Sympos. Pure Math., Oregon State Univ., Corvallis, Ore., 1977), Part
      2},
   },
   book={
      publisher={Amer. Math. Soc.},
      place={Providence, R.I.},
   },
   date={1979},
   pages={247--289},
   review={\MR{546620 (81i:10032)}},
}


\bib{Fal84}{article}{
   author={Faltings, Gerd},
   title={Finiteness theorems for abelian varieties over number fields},
   note={Translated from the German original [Invent.\ Math.\ {\bf 73}
   (1983), no.\ 3, 349--366; ibid.\ {\bf 75} (1984), no.\ 2, 381; MR
   85g:11026ab] by Edward Shipz},
   conference={
      title={Arithmetic geometry},
      address={Storrs, Conn.},
      date={1984},
   },
   book={
      publisher={Springer, New York},
   },
   date={1986},
   pages={9--27},
   review={\MR{861971}},
}

\bib{Huy14}{article}{
   author={Huybrechts, Daniel},
   title={Lectures on K3 surfaces},
   note={Lecture notes},
   date={2014},
   eprint={http://www.math.uni-bonn.de/people/huybrech/K3Global.pdf}
   }
   
   
\bib{KS67}{article}{
   author={Kuga, Michio},
   author={Satake, Ichir{\^o}},
   title={Abelian varieties attached to polarized $K_{3}$-surfaces},
   journal={Math. Ann.},
   volume={169},
   date={1967},
   pages={239--242},
   issn={0025-5831},
   review={\MR{0210717 (35 \#1603)}},
}


\bib{Lan13}{book}{
   author={Lan, Kai-Wen},
   title={Arithmetic compactifications of PEL-type Shimura varieties},
   series={London Mathematical Society Monographs Series},
   volume={36},
   publisher={Princeton University Press, Princeton, NJ},
   date={2013},
   pages={xxvi+561},
   isbn={978-0-691-15654-5},
   review={\MR{3186092}},
}

\bib{LMS14}{article}{
   author={Lieblich, Max},
   author={Maulik, Davesh},
   author={Snowden, Andrew},
   title={Finiteness of K3 surfaces and the Tate conjecture},
   language={English, with English and French summaries},
   journal={Ann. Sci. \'Ec. Norm. Sup\'er. (4)},
   volume={47},
   date={2014},
   number={2},
   pages={285--308},
   issn={0012-9593},
   review={\MR{3215924}},
}


\bib{Mil86}{article}{
   author={Milne, J. S.},
   title={Abelian varieties},
   conference={
      title={Arithmetic geometry},
      address={Storrs, Conn.},
      date={1984}, },
   book={publisher={Springer, New York}, },
   date={1986},
   pages={103--150},
   review={\MR{861974}},
}

\bib{Mil05}{article}{
   author={Milne, J. S.},
   title={Introduction to Shimura varieties},
   conference={
      title={Harmonic analysis, the trace formula, and Shimura varieties},
   },
   book={
      series={Clay Math. Proc.},
      volume={4},
      publisher={Amer. Math. Soc., Providence, RI},
   },
   date={2005},
   pages={265--378},
   review={\MR{2192012 (2006m:11087)}},
}

\bib{Moo98}{article}{
   author={Moonen, Ben},
   title={Models of Shimura varieties in mixed characteristics},
   conference={
      title={Galois representations in arithmetic algebraic geometry
      (Durham, 1996)},
   },
   book={
      series={London Math. Soc. Lecture Note Ser.},
      volume={254},
      publisher={Cambridge Univ. Press, Cambridge},
   },
   date={1998},
   pages={267--350},
   review={\MR{1696489 (2000e:11077)}},
   doi={10.1017/CBO9780511662010.008},
}




\bib{MP15}{article}{
   author={Madapusi Pera, Keerthi},
   title={The Tate conjecture for K3 surfaces in odd characteristic},
   journal={Invent. Math.},
   volume={201},
   date={2015},
   number={2},
   pages={625--668},
   issn={0020-9910},
   review={\MR{3370622}},
   doi={10.1007/s00222-014-0557-5},
}


\bib{MP16}{article}{
   author={Madapusi Pera, Keerthi},
   title={Integral canonical models for spin Shimura varieties},
   journal={Compos. Math.},
   volume={152},
   date={2016},
   number={4},
   pages={769--824},
   issn={0010-437X},
   review={\MR{3484114}},
   doi={10.1112/S0010437X1500740X},
}


\bib{Ogu83}{article}{
   author={Ogus, Arthur},
   title={A crystalline Torelli theorem for supersingular $K3$ surfaces},
   conference={
      title={Arithmetic and geometry, Vol. II},
   },
   book={
      series={Progr. Math.},
      volume={36},
      publisher={Birkh\"auser Boston, Boston, MA},
   },
   date={1983},
   pages={361--394},
   review={\MR{717616 (85d:14055)}},
}


\bib{PSS71}{article}{
   author={Pjatecki{\u\i}-{\v{S}}apiro, I. I.},
   author={{\v{S}}afarevi{\v{c}}, I. R.},
   title={Torelli's theorem for algebraic surfaces of type ${\rm K}3$},
   language={Russian},
   journal={Izv. Akad. Nauk SSSR Ser. Mat.},
   volume={35},
   date={1971},
   pages={530--572},
   issn={0373-2436},
   review={\MR{0284440 (44 \#1666)}},
}
\bib{Riz05}{thesis}{
   author={Rizov, Jordan},
   title={Moduli of  K3  Surfaces and Abelian Varieties},
   date={2005},
   eprint={http://dspace.library.uu.nl/bitstream/handle/1874/7290/full.pdf?sequence=1},
   type={Ph.D. Thesis}
   organization={Universiteit Utrecht}
}

\bib{Riz06}{article}{
   author={Rizov, Jordan},
   title={Moduli stacks of polarized  K3  surfaces in mixed characteristic},
   journal={Serdica Math. J.},
   volume={32},
   date={2006},
   number={2-3},
   pages={131--178},
   issn={1310-6600},
}

\bib{Riz10}{article}{
   author={Rizov, Jordan},
   title={Kuga-Satake abelian varieties of K3 surfaces in mixed
   characteristic},
   journal={J. Reine Angew. Math.},
   volume={648},
   date={2010},
   pages={13--67},
   issn={0075-4102},
   doi={10.1515/CRELLE.2010.078},
}
\bib{Sha96}{book}{
title={Algebraic Geometry II},
subtitle={Cohomology of Algebraic Varieties. Algebraic Surfaces},
author={Shafarevich, I.R.},
translator={Treger, R.},
publisher={Springer-Verlag},
date={1996}
isbn={978-3540546801}
}

\bib{Tat94}{article}{
   author={Tate, John},
   title={Conjectures on algebraic cycles in $l$-adic cohomology},
   conference={
      title={Motives},
      address={Seattle, WA},
      date={1991},
   },
   book={
      series={Proc. Sympos. Pure Math.},
      volume={55},
      publisher={Amer. Math. Soc.},
      place={Providence, RI},
   },
   date={1994},
   pages={71--83},
   review={\MR{1265523 (95a:14010)}},
}

\bib{Zar85}{article}{
   author={Zarhin, Yu. G.},
   title={A finiteness theorem for unpolarized abelian varieties over number
   fields with prescribed places of bad reduction},
   journal={Invent. Math.},
   volume={79},
   date={1985},
   number={2},
   pages={309--321},
   issn={0020-9910},
   review={\MR{778130 (86d:14041)}},
   doi={10.1007/BF01388976},
}



\end{biblist}
\end{document}